\documentclass[11pt]{amsart}
\usepackage[left=1in,top=1.3in,right=1in,bottom=1in,head=.3in,foot=.3in]{geometry}
\usepackage{tikz-cd, amssymb}

\newcommand{\F}{\mathbb F}
\newcommand{\K}{\mathcal K}
\newcommand{\m}{\mathfrak m}
\newcommand{\N}{\mathbb N}
\newcommand{\cO}{\mathcal O}
\newcommand{\Q}{\mathbb Q}
\newcommand{\R}{\mathbb R}
\newcommand{\Z}{\mathbb Z}

\newcommand{\alg}{\text{alg}}
\DeclareMathOperator{\Aut}{Aut}
\DeclareMathOperator{\Char}{char}
\DeclareMathOperator{\res}{res}
\newcommand{\sep}{\text{sep}}

\newcommand{\Ldiv}{{\mathcal L_{\operatorname{div}}}}
\newcommand{\Lring}{{\mathcal L_{\operatorname{ring}}}}

\theoremstyle{plain}
\newtheorem{theorem}{Theorem}[section]
\newtheorem{lemma}[theorem]{Lemma}
\newtheorem{prop}[theorem]{Proposition}
\newtheorem{cor}[theorem]{Corollary}
\newtheorem{conj}[theorem]{Conjecture}
\newtheorem{fact}[theorem]{Fact}
\newtheorem{claim}[theorem]{Claim}

\theoremstyle{definition}
\newtheorem{defn}[theorem]{Definition}
\newtheorem{example}[theorem]{Example}

\newtheorem{quest}[theorem]{Question}

\begin{document}

\title{Immediately algebraically closed fields}
\author{Peter Sinclair}
\address{Peter Sinclair, Department of Mathematics, Douglas College, 700 Royal Ave, New Westminster, British Columbia V3M 5Z5, Canada}
\email{sinclapt@mcmaster.ca}

\begin{abstract}
    We consider two overlapping classes of fields, IAC and VAC, which are defined using valuation theory but which do not involve a distinguished valuation. Rather, each class is defined by a condition that quantifies over all possible valuations on the field. In his thesis, Hong asked whether these two classes are equal \cite[Question 5.6.8]{Hong13}. In this paper, we give an example that negatively answers Hong's question. We also explore several situations in which the equivalence does hold with an additional assumption, including the case where every $K'\equiv K$ is IAC.
\end{abstract}

\maketitle

\begin{center}
\begin{minipage}{.7\textwidth}
\setcounter{tocdepth}{1}
\tableofcontents
\end{minipage}
\end{center}

\section{Introduction}

The model theory of fields with a single distinguished valuation is well established. Fields with several distinguished valuations have also been considered, for example in \cite{Ers01}, \cite{John16}, and \cite{Mon17}. In this paper, we consider two overlapping classes of fields which are defined using valuation theory, but which do not involve any distinguished valuations; each class is instead defined by a condition that quantifies over all possible valuations on the field. From a model theoretic perspective, we consider both classes in the language of fields, rather than in a language of valued fields.

\begin{defn}
    We say that a field $K$ is \emph{immediately algebraically closed} (IAC) if, for every non-trivial valuation $v$ on $K$, $Kv$ is algebraically closed and $vK$ is divisible. We say that a field $K$ is \emph{valuationally algebraically closed} (VAC) if, for every non-trivial valuation $v$ on $K^\alg$, $K$ is dense in its algebraic closure with respect to $v$.
\end{defn}

These definitions are given in Hong's doctoral thesis \cite{Hong13}, where he suggested VAC in particular may be useful as an intermediate step in proving the stable field conjecture (this relationship is discussed further in Section \ref{52-stronglyiac}). They were independently considered in \cite{Krup15}, where it is shown that every superrosy field of positive characteristic is IAC.

Immediately algebraically closed fields also appear under purely algebraic assumptions. It is easy to see that every algebraically closed field is both VAC and IAC. In fact, this is also true for every separably closed field (Proposition \ref{separably}, below) and every pseudo-algebraically closed field \cite[Proposition 11.5.3]{FJ08}.

The main result of this paper is a partial answer to the following question posed by Hong:

\begin{quest}
    Suppose $K$ is an IAC field. Is $K$ also VAC?
\end{quest}

In Section \ref{counterexample}, we answer this question negatively by constructing a henselian subfield of $\F_p(t)^\alg$ that is IAC but is not dense in its algebraic closure with respect to any extension of its henselian valuation. On the other hand, Section \ref{algebra} discusses three algebraic conditions on an IAC field $K$ that imply $K$ is VAC:
\begin{itemize}
	\item $K$ has positive characteristic and no proper Artin-Schreier extensions (Theorem \ref{asc}),
	\item $K$ has characteristic zero and its multiplicative group is divisible (Theorem \ref{divisible}),
	\item $K$ is real closed and archimedean (Theorem \ref{realthm}).
\end{itemize}
Each of these conditions are sufficient, but we believe that they are all stronger than necessary. In the final Section \ref{52-stronglyiac}, we show that IAC and VAC are equivalent if they are considered as conditions on the theory of a field, rather than a particular model.

This paper is based on results from a chapter of the author's thesis \cite{Sin18}, under the supervision of Professor Deirdre Haskell.

\section{Basic Notions}
\subsection{Valuations}

We assume that the reader is familiar with the basic notions of valued fields. For more detail, refer to any textbook on valued fields, such as \cite{EP05}. Given a valuation $v$ on a field $K$, we denote the value group by $vK$, the residue field by $Kv$, the valuation ring by $\cO_v$, and the maximal ideal by $\m_v$.

When considering extensions of valued fields $L/K$, we will write $v$ for both the valuation on $L$ and its restriction to $K$. If $L/K$ is an algebraic extension then $vK$ and $vL$ will have the same divisible hull and $Lv$ will be an algebraic extension of $Kv$. In particular, if $L = K^\alg$ then $vL$ will be the divisible hull of $vK$ and $Lv$ will be the algebraic closure of $Kv$.

Recall that every valuation induces a topology generated by the basic open sets
\[ B(a,\gamma) = \{x\in K : v(x-a) > \gamma\} . \]
We say that a subset $A$ is dense in a field $K$ if for every $b\in K$ and $\gamma\in vK$, there exists $a\in A$ such that $a\in B(b,\gamma)$, or equivalently, $v(a-b)>\gamma$. We can similarly define Cauchy and convergent sequences:
\begin{itemize}
	\item A sequence $(a_\alpha)_{\alpha < \kappa}$ is Cauchy if for all $\gamma\in vK$ there exists $\beta<\kappa$ such that $\alpha,\alpha'\geq \beta$ implies $v(a_\alpha - a_{\alpha'}) > \gamma$.
	\item A sequence $(a_\alpha)_{\alpha < \kappa}$ converges to an element $b\in K$ if for all $\gamma\in vK$ there exists $\beta<\kappa$ such that $\alpha\geq \beta$ implies $v(a_\alpha - b) > \gamma$.
\end{itemize}
As with ordered fields, each valued field $(K,v)$ has a unique minimal extension in which every Cauchy sequence is convergent; we call this field the completion of $K$ with respect to $v$.

\begin{defn}
	Consider two valuations $v$ and $w$ on a field $K$. We say that $w$ is a coarsening of $v$ if $\cO_w \supseteq \cO_v$; in this case we also say that $\cO_w$ is a coarsening of $\cO_v$.
\end{defn}

There is a one-to-one order-preserving correspondence of the coarsenings of a valuation with the convex subgroups of the value group.

\begin{defn}
    A subgroup $\Delta$ of an ordered abelian group $\Gamma$ is said to be \emph{convex} if for every $a\in \Delta$, the interval $[-a,a] = \{x\in\Gamma : -a\leq x\leq a\}$ is a subset of $\Delta$. The convex subgroups of $\Gamma$ are linearly ordered by inclusion, and this order type is called the \emph{rank} of $\Gamma$. In particular, if $\Gamma$ is non-trivial and has no proper non-trivial convex subgroups then $\Gamma$ has rank 1, and is called \emph{archimedean}.
\end{defn}

Given a valued field $(K,v)$ and a convex subgroup $\Delta \leq vK$, we can define
\[ \cO_\Delta = v^{-1}(\Delta)\cup\cO = \{x\in K : v(x) \geq \delta\text{ for some }\delta\in\Delta\} .\]
Clearly, $\cO_\Delta \supseteq \cO$ is a valuation ring; it uniquely defines a valuation $w_\Delta: K^\times \to vK/\Delta$, which is a coarsening of $v$. Conversely, given a coarsening $w$ of a valuation $v$, the set
\[ \Delta_w = \{v(x) \in vK : x\in K \text{ and } w(x) = 0 \} \]
is a convex subgroup of $vK$. It is easy to check that this correspondence, when viewed between valuation rings and convex subgroups, is order-preserving. One immediate and useful consequence of this correspondence is that the coarsenings of a valuation ring are linearly ordered, just like the convex subgroups of the value group.

In general, fields have infinitely many possible valuations. In certain cases, this collection of valuations has enough structure that we can still describe the complete set of valuations.

\begin{example} \label{rationalvaluations}
	Let $K$ be any field equipped with the trivial valuation and consider the field $K(t)$ of rational functions over $K$. By Theorem 2.1.4 of \cite{EP05}, every non-trivial valuation on $K(t)$ is one of the following:
	\begin{itemize}
		\item The degree valuation $v_\infty: K(t) \to \Z$ defined by
		\[ v_\infty\left(\frac{f}{g}\right) = \deg(g) - \deg(f) \]
		for any polynomials $f,g\in K[t]$.
		\item An $f$-adic valuation $v_f: K(t)\to \Z$, which is defined in the same way as the $p$-adic valuation on $\Q$, using an irreducible polynomial $f\in K[t]$ in place of the prime $p$. To be precise, fix $f$ and let $r\in K[t]$. Then there exist polynomials $g, h\in K[t]$ such that $f$ does not divide either $g$ or $h$ and
        \[ r = f^n\left(\frac{g}{h}\right) \]
        for some $n\in\Z$. We define $v_f(r)$ to be this integer $n$.
	\end{itemize}

	Each of these valuations has value group $\Z$. The degree valuation has residue field isomorphic to $K$ and the residue field of an $f$-adic valuation is a finite extension of $K$, specifically $K[t]/(f)$.
	
	Now consider the particular case where $K$ is the algebraic closure of $\F_p$, the finite field with $p$ elements. Because non-trivial ordered abelian groups must be infinite, the only valuation on $\F_p$ is the trivial valuation $v(x) = 0$ for all $x$. Since algebraic extensions of a valued field cannot increase the rank of the value group, the trivial valuation is also the only valuation on $K$.
	
	Then by the above argument, the only valuations on $K(t)$ are the degree valuation and the $f$-adic valuations, all of which have value group $\Z$. Moreover, since $K$ is algebraically closed, each of these valuations has residue field isomorphic to $K$. We will revisit this example in Section \ref{counterexample}.
\end{example}

\subsection{IAC and VAC}

Recall the definition of IAC from the introduction:

\begin{defn}
    We say that a field $K$ is \emph{immediately algebraically closed} (IAC) if, for every non-trivial valuation $v$ on $K$, $Kv$ is algebraically closed and $vK$ is divisible.
\end{defn}

There are several other ways we could choose to define IAC:

\begin{prop} \label{equiv}
    The following are equivalent:
    \begin{enumerate}
        \item $K$ is IAC.
        \item For every non-trivial valuation $v$ on $K^\alg$, $vK$ is divisible, and if $a\in K^\alg$ with $v(a) = 0$ then there exists $b\in K$ with $v(b) = 0$ and $\res(a) = \res(b)$.
        \item For all $a\in K^\alg$ and every non-trivial valuation $v$ on $K^\alg$, there exists $b\in K$ with $v(a-b) > v(a)$.
        \item $K^\alg$ is an immediate extension of $K$ with respect to any non-trivial valuation.
    \end{enumerate}
\end{prop}

These equivalences all follow directly from the definition of IAC. Condition (4) above explains the origin of the name \emph{immediately} algebraically closed. If one thinks of immediate extensions in terms of pseudo-convergent sequences (as in \cite{Kap42}), this tells us that an IAC field is ``pseudo-dense'' in its algebraic closure. Similarly, conditions (2) and (3) can be interpreted to say that an IAC field can approximate elements in the algebraic closure somewhat well. This leads us to the definition of VAC fields (repeated from the introduction) and suggests a possible equivalence between IAC and VAC.

\begin{defn}
    We say that a field $K$ is \emph{valuationally algebraically closed} (VAC) if, for every non-trivial valuation $v$ on $K^\alg$, $K$ is dense in its algebraic closure with respect to $v$.
\end{defn}

The following results are immediate from the definitions of IAC and VAC.

\begin{prop} \label{easy}
    Suppose $K$ is a field.
    \begin{enumerate}
        \item Every algebraic extension of an IAC field is IAC.
        \item Every algebraic extension of a VAC field is VAC.
        \item If $K$ is VAC then $K$ is IAC.
    \end{enumerate}
\end{prop}

We conclude this section by briefly discussing the relationship between VAC fields and separably closed fields, which we will return to at the end of Section \ref{52-stronglyiac}. First, we recall the continuity of roots in valued fields:

\begin{fact} \label{continuity}
	Let $(K,v)$ be a valued field, let $\alpha\in vK$, and let $f(X) = \sum_{i=0}^n a_iX^i$ be a polynomial in $K[X]$ with distinct roots $x_1,\ldots,x_n\in K$. Then there exists $\gamma\in vK$ such that for all $y_1,\ldots,y_n\in K$ with
	\[ g(X) = \prod_{i=1}^n (X-y_i) = \sum_{i=0}^n b_iX^i \]
	and $\min v(a_i-b_i) > \gamma$, for every $x_i$ there is at least one $y_j$ with $v(x_i-y_j)>\alpha$. Moreover, if $\alpha \geq v(x_i-x_j)$ for all $i\neq j$ then there exists exactly one $y_j$ with $v(x_i-y_j) > \alpha$.
\end{fact}

This is stated and proved as Theorem 2.4.7 of \cite{EP05}. Continuity of roots essentially says that a small variation in the coefficients of a polynomial results in a small variation in its roots. This can be combined with Lemma 3.2.12 of \cite{EP05}, which essentially states that the set of separable polynomials is dense in the set of all polynomials, to prove the following:

\begin{prop} \label{separably}
	Suppose $K$ is a separably closed field. Then $K$ is both IAC and VAC.
\end{prop}

\begin{proof}
	Fix $c\in K^\alg$, a valuation $v$ on $K^\alg$, and $\alpha \in vK^\alg$. Let $f(X) = \sum_{i=0}^n a_iX^i$ be the minimal polynomial for $c$ and choose $\gamma$ as in continuity of roots. By \cite[Lemma 3.2.12]{EP05}, we can find a separable polynomial $g(X) = \sum_{i=0}^n b_iX^i \in K[X]$ with $v(a_i-b_i) > \gamma$ for all $i<n$. Then by continuity of roots, $g(X)$ has a root $d$ with $v(c-d) > \alpha$, and since $K$ is separably closed, $d\in K$. Hence $K$ is dense in $K^\alg$ with respect to $v$, which means $K$ is VAC. The fact that $K$ is IAC then follows immediately from Proposition \ref{easy}.
\end{proof}

This result can also be seen as a consequence of \cite[Lemma 1.6.2]{Ers01}. Because we know that pseudo-algebraically closed fields are also VAC, the converse of this result is false. However, we can obtain a partial converse by adding henselianity as an assumption.

\begin{prop} \label{no-henselian}
    Suppose $K$ is a VAC field. If there exists a non-trivial henselian valuation $v$ on $K$ then $K$ is separably closed.
\end{prop}

\begin{proof}
    Let $v$ also denote its unique extension to $K^\alg$, and fix $a\in K^\sep$. Since $K$ is VAC, there exists $b\in K$ with $v(b-a) > v(a'-a)$ for every conjugate $a'$ of $a$ over $K$. Then by Krasner's lemma (see \cite[Theorem 4.1.7]{EP05}), $a \in K(b) = K$.
\end{proof}

This proposition shows two interesting things. One, the class of VAC fields is in this sense orthogonal to the class of henselian fields. Two, even though VAC is on its surface a topological property, it does have significant algebraic consequences. As we will show in the following section, this orthogonality with henselianity does not hold for IAC fields.

\section{Counterexamples}
\label{counterexample}

Recall that a polynomial of the form $X^p-X-a$ with $p = \Char(K) > 0$ is called an Artin-Schreier polynomial and that a field extension $L/K$ is called an Artin-Schreier extension if $L$ is generated over $K$ by the root of an Artin-Schreier polynomial over $K$. Note that Artin-Schreier extensions are always Galois: they are clearly separable, and if $\theta$ is a root of an Artin-Schreier polynomial, then the full set of roots is $\{\theta,\theta+1,\ldots,\theta+p-1\}$.

The main result of this section is an example of a field of positive characteristic that is IAC but not VAC. In particular, we will construct a field that is IAC and henselian but that has a proper Artin-Schreier extension; the field is then not separably closed, and so cannot be VAC by Proposition \ref{no-henselian}. First we state the following fact, due to Quigley:

\begin{fact} \label{Quigley}
    \cite[Theorem 1]{Quig} Let $K$ be a field and fix $\alpha \in K^\alg$. Let $M$ be a subfield of $K^\alg$ that contains $K$ and is maximal with respect to the property $\alpha\notin M$. Then there exists a prime $q$ such that:
    \begin{enumerate}
        \item $[N:M]$ is a power of $q$ for every finite normal extension $N$ of $M$.
        \item Either $M$ is perfect or $K^\alg$ is a purely inseparable extension of $M$.
        \item $[M(\alpha):M] = q$ and $M(\alpha)$ is a normal extension of $M$.
        \item $M$ contains all $q$th roots of unity.
    \end{enumerate}
\end{fact}

\begin{example}
    Let $K$ be the algebraic closure of $\F_p$, the finite field with $p$ elements. As observed in Example \ref{rationalvaluations}, every non-trivial valuation on the field $K(t)$ of rational functions over $K$ is either the degree valuation $v_\infty$ or an $f$-adic valuation for some irreducible $f\in K[t]$. Let $(L,v)$ be a henselization of $(K(t), v_\infty)$.
	
	Let $\theta\in L^\alg = K(t)^\alg$ be a root of the Artin-Schreier polynomial $X^p-X-t^{-1}$. Since $L$ is henselian, there is a unique extension of $v$ to $L(\theta)$, which we also denote by $v$. Then we have $v(\theta) = -p^{-1}\notin \Z = v_\infty K(t)$; the henselization is always an immediate extension, so $vL = v_\infty K(t)$ and hence $\theta\notin L$.
	
    By a straightforward Zorn's Lemma argument, there is a subfield $M$ of $L^\alg$ which contains $L$ and is maximal with respect to the property $\theta\notin M$. This field is henselian, because it is an algebraic extension of a henselian field.
	
	
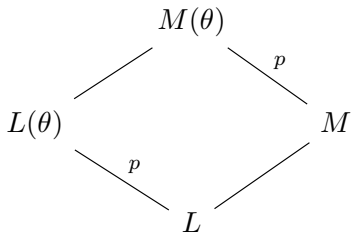
\begin{center}
		\begin{tikzcd}
		& M(\theta) \arrow[dl,dash]\arrow[dr,dash,"p"] & \\
		L(\theta) \arrow[dr,dash,"p"] & & M\arrow[dl,dash] \\
		& L &
		\end{tikzcd}
	\end{center}	
	
	By Fact \ref{Quigley}, since $M(\theta)$ is a separable extension of $M$, $M$ is perfect. Moreover, the prime $q$ in Fact \ref{Quigley} is $p = \Char(K)$ because $[M(\theta):M] =p$, and hence $[N:M]$ is a power of $p$ for every finite normal extension $N$ of $M$.
	
	\begin{claim}
		The multiplicative group $M^\times$ is divisible.
	\end{claim}
	
	\begin{proof}
		If $c\in M^\alg$ with $c^p\in M$ then, because $M$ is perfect, we must have $c\in M$. On the other hand, suppose $c\in M^\alg$ with $c^q\in M$ for some prime $q\neq p$. If $c\notin M$ then the splitting field of $X-c^q$ is a finite normal extension, so its degree is a power of $p$, again by Fact 3.1. But $[M(c):M]=q$ does not divide any power of $p$. By contradiction, we must have $c\in M$, and hence $M^\times$ is divisible.
	\end{proof}
	
	It follows immediately from the claim that for any non-trivial valuation $w$ on $M$, the value group is divisible. Moreover, as observed in Example \ref{rationalvaluations}, the residue field of $K(t)$ with respect to any valuation is isomorphic to $K$. Since $M$ is an algebraic extension of $K(t)$, it follows that the residue field $Mw$ for any valuation $w$ on $M$ will also be isomorphic to $K$. Thus, $M$ is IAC and henselian. But $M$ is not separably closed because $M(\theta)$ is a proper separable extension, and hence $M$ cannot be VAC by Proposition \ref{no-henselian}.
\end{example}

We could have concluded that $M$ is not VAC without Proposition \ref{no-henselian} by instead using the fact that there exists a valuation on $M$ that extends uniquely to an Artin-Schreier extension of $M$. Recall the definition of the defect of a valued field extension:

\begin{defn}
    Let $N/K$ be a Galois extension, and fix a valuation $v$ on $N$. Let $e = [vN:vK]$ and $f = [Nv:Kv]$, and let $r$ be the number of distinct valuations $v'$ on $N$ with $v'|_K = v$. The \emph{defect} of $(N,v)/(K,v)$ is the positive integer
    \[ d = \frac{[N:K]}{ref} .\]
    The extension $(N,v)/(K,v)$ is called a \emph{defect extension} if $d>1$, and \emph{defectless} if $d=1$. See Section 3.3 of \cite{EP05} for more details.
\end{defn}

The defect plays an important role in the Galois theory of valued fields, but in the case of Artin-Schreier extensions of IAC fields, it simply measures whether the valuation extends uniquely. More precisely, if $L$ is a proper Artin-Schreier extension of an IAC field $K$ of characteristic $p>0$ then $[L:K]=p$ and $e=f=1$, so $r=1$ if and only if $d\neq 1$.

\begin{prop}\label{no-asc}
    Suppose $K$ is a VAC field of positive characteristic, and fix a valuation $v$ on $K$. Then $(K,v)$ has no Artin-Schreier defect extensions.
\end{prop}

\begin{proof}
    Suppose $L$ is an Artin-Schreier defect extension of $(K,v)$ and $\theta$ is the root of an Artin-Schreier polynomial with $L = K(\theta)$; then there is a unique extension of $v$ to $L$, which we also denote by $v$. By Corollary 2.30 of \cite{Kuhl10}, $v(\theta-c) < 0$ for all $c\in K$, which means $K$ is not dense in $K(\theta)$. But $K$ is dense in $K^\alg$, which means it must be dense in every algebraic extension of $K$; by contradiction, no such $L$ can exist.
\end{proof}

\begin{example}
	Consider $M$ and $\theta$ from the previous example, and let $v$ be a non-trivial henselian valuation on $M$. Since $v$ extends uniquely to $M(\theta)$, $M(\theta)$ is an Artin-Schreier defect extension of $(M,v)$. Thus, we also could have concluded that $M$ is not VAC using Proposition \ref{no-asc}.
\end{example}

It is not known whether Propositions \ref{no-henselian} and \ref{no-asc} can fail in an IAC field of positive characteristic independently of each other; that is, whether there exists a non-henselian IAC field with an Artin-Schreier defect extension or a henselian IAC field without an Artin-Schreier defect extension. However, in Section \ref{positive}, we will show that eliminating all Artin-Schreier extensions (with or without defect) is enough to ensure the equivalence of IAC and VAC.

In a private communication, Philip Dittmann pointed out to the author that it is also possible to construct an IAC field of characteristic zero that is not VAC, as described in the following example.

\begin{example}
	Let $(\Q_p,v_p)$ be the $p$-adic numbers equipped with the standard $p$-adic valuation and consider an algebraic extension $(K,v)$ of $(\Q_p,v_p)$ such that $Kv$ is algebraically closed, $vK$ is divisible, and $K$ is not algebraically closed. One example of such a field is the maximal tamely ramified extension of $\Q_p$, which is not algebraically closed by Theorem 2(ii) of \cite{Iwa55}.
	
	
	
	Let $w$ be any valuation on $K$ and consider the coarsening $\cO_v\cO_w$ of $\cO_v$. Since $(K,v)$ has rank 1, $\cO_v$ has no non-trivial proper coarsenings; hence $\cO_v\cO_w$ must equal $\cO_v$ or $K$.
		
	If $\cO_v\cO_w = \cO_v$ then $\cO_w\subseteq \cO_v$; in other words, $v$ is a coarsening of $w$. Then there must be a convex subgroup $\Delta$ of $wK$ with $vK\cong wK/\Delta$; moreover, $w$ induces a valuation on $Kv$ with value group $\Delta$. Since $Kv$ is an algebraic extension of a finite field, $\Delta$ must be trivial, and so $v$ and $w$ are equivalent.
	
	On the other hand, if $\cO_v\cO_w = K$ then $v$ and $w$ are independent valuations. Let $(K^h,w^h)$ be the henselization of $(K,w)$. Then $K^h$ has two independent henselian valuations: $w^h$ and the unique extension of $v$ to $K^h$. Because $K$ has characteristic zero, it follows form \cite[Theorem 4.4.1]{EP05} that $K^h$ is algebraically closed, and so $Kw = K^hw^h$ is algebraically closed and $wK = w^hK^h$ is divisible.
	
	Thus, $K$ is IAC. But $(K,v)$ is henselian and $K$ is not separably closed, so by Proposition \ref{no-henselian}, $K$ is not VAC.
\end{example}	

\section{Algebraic Conditions}
\label{algebra}

In this section, we provide sufficient algebraic conditions to deduce that an IAC field is VAC. Before we prove the main results of the section, we establish a criterion for verifying that a field is VAC.

\subsection{A VAC Criterion} Recall that every valued field has a unique minimal extension in which every Cauchy sequence is convergent, called its completion. We begin by establishing some relationships between completions and algebraic closures.

\begin{prop} \label{prop-0.1}
	Suppose $(K,v)$ is a valued field. The following are equivalent:
	\begin{enumerate}
		\item The completion $(L,w)$ of $(K,v)$ is separably closed.
		\item The completion $(L,w)$ of $(K,v)$ is algebraically closed.
		\item There exists a valuation $w$ on $K^\alg$ which extends $v$ such that $(K,v)$ is dense in $(K^\alg,w)$.
		\item For every valuation $w$ on $K^\alg$ which extends $v$, $(K,v)$ is dense in $(K^\alg,w)$.
		\item For every non-constant separable $f(X)\in K[X]$ and $\gamma\in vK$, there is $x\in K$ such that $v(f(x)) > \gamma$.
	\end{enumerate}
\end{prop}

\begin{proof}
    $(1)\to (2)$: Every separably closed field is VAC by Proposition \ref{separably}. But then $(L,w)$ is dense in its algebraic closure and complete, so $L = L^\alg$.
	
	$(2)\to (3)$: If $L$ is algebraically closed then by the universal property of algebraic closures, we can identify $K^\alg$ with a subfield of $L$ containing $K$. The restriction of $w$ to $K^\alg$ extends $v$ and $(K,v)$ is dense in $(K^\alg,w) \subseteq (L,w)$.
	
	$(3)\to (4)$: Suppose $w$ is an extension of $v$ to $K^\alg$ such that $(K,v)$ is dense in $(K^\alg,w)$, and let $w'$ be some other extension of $v$ to $K^\alg$. Then by the Conjugation Theorem \cite[3.2.15]{EP05}, there exists $\sigma\in \Aut(K^\alg/K)$ which induces a valued field isomorphism $(K^\alg,w)\to (K^\alg,w')$ over $(K,v)$. Thus $(K,v)$ is also dense in $(K^\alg,w')$.
	
	
	$(4)\to (5)$: Fix a non-constant separable polynomial $f(X) \in K[X]$ and $\gamma\in vK$. Let $b_0,\ldots,b_n\in K^\alg$ be the roots $f(X)$ and let $\delta = \min_{i\neq 0}\{v(b_0-b_i)\}$. Without loss of generality, assume that $\gamma > v(b_0-b_i) + n\delta$ for all $1\leq i\leq n$. Since $K$ is dense in $K^\alg$, there exists $x\in K$ with $v(x-b_0) > \gamma - n\delta > v(b_0-b_i)$ for all $1\leq i\leq n$. Then $v(x-b_i) = v(b_0-b_i) \geq \delta$ for all $1\leq i\leq n$ and
	\[ v(f(x)) = v(x-b_0) + v(x-b_1) + \ldots + v(x-b_n) > (\gamma-n\delta) + \delta + \ldots +\delta = \gamma \]
	as desired.
	
	
	$(5)\to(1)$: Let $(L,v)$ be the completion of $(K,v)$; we will work in the algebraic closure $(L^\alg,v)$. Fix a cofinal sequence $(\gamma_\alpha)_{\alpha<\kappa}$ in $vK=vL$ and let $f(X)\in L[X]$ be a separable polynomial of degree $n>0$. We will show that $f(X)$ has a root in $L$ by finding a convergent sequence in $K$.
	
	By continuity of roots (Fact \ref{continuity}), for each $\gamma_\alpha$ there exists $\delta_\alpha\in vK$ such that for any $g(X)\in L[X]$ with $\deg(g) = n$, if the coefficients of $g$ and $f$ are within $\delta_\alpha$ of each other then every root of $g$ is within $\gamma_\alpha$ of a root of $f$. Since $K$ is dense in $L$, we can choose a separable polynomial $g_\alpha(X)\in K[X]$ with $\deg(g_\alpha) = n$ and the coefficients of $g_\alpha$ and $f$ within $\delta_\alpha$ of each other.
	
	Because $g_\alpha(X)$ is separable, by (5) there is $x_\alpha\in K$ such that $g_\alpha(x_\alpha) > n\gamma_\alpha$. Then by the pigeonhole principle, $v(x_\alpha-b) > \gamma_\alpha$ for some $b\in L^\alg$ with $g_\alpha(b) = 0$. By choice of $g_\alpha$, we also have $v(b-c) > \gamma_\alpha$ for some $c\in L^\alg$ with $f(c) = 0$. But then
	\[ v(x_\alpha-c) \geq \min\{v(x_\alpha-b), v(b-c)\} > \gamma_\alpha .\]
	While it is not necessarily true that each $x_\alpha$ approximates the same root $c$, by the pigeonhole principle again, there is a cofinal subsequence of $(x_\alpha)_{\alpha<\kappa}$ and a single root $d$ of $f(X)$ such that each $x_\alpha$ in the subsequence satisfies $v(x_\alpha-d)>\gamma_\alpha$. Then this subsequence converges to $d$, which means $d\in L$, and hence $L$ is separably closed.
\end{proof}

The equivalence of (3) and (4) in the proposition shows us that for a valued field $(K,v)$, density in the algebraic closure depends only on $v$ and not on its extension to $K^\alg$. Thus, we will frequently say that $K$ is dense in $K^\alg$ with respect to $v$, rather than the more precise ``with respect to an extension of $v$ to $K^\alg$.''

Consider a field $K$ with two valuations $v$ and $w$, and recall that $\cO_w$ is called a coarsening of $\cO_v$ if $\cO_v\subseteq\cO_w$. In this case, $v$ and $w$ are dependent valuations, and so by Theorem 2.3.4 of \cite{EP05}, they induce the same topology. Moreover, the correspondence between coarsenings of $\cO_v$ and convex subgroups of $vK$ implies that the set of coarsenings of $\cO_v$ is linearly ordered by inclusion. Determining whether an IAC field is VAC depends only on the valuations that have a maximum non-trivial coarsening, as the following lemma shows.

\begin{lemma} \label{nomax}
    Suppose $(K,v)$ is a valued field such that $\cO_v$ has no maximum non-trivial coarsening and such that every non-trivial coarsening of $\cO_v$ has algebraically closed residue field. Then $K$ is dense in $K^\alg$ with respect to $v$.
\end{lemma}

\begin{proof}
	Fix a valuation $w$ on $K^\alg$ extending $v$. For each $\gamma \in vK$, let $\Delta_\gamma$ be the smallest convex subgroup of $vK$ containing $\gamma$, and let $v_\gamma$ be the coarsening of $v$ corresponding to $\Delta_\gamma$. Then for all $x\in K$, $v_\gamma(x)>0$ if and only if $v(x) > n\gamma$ for all $n\in\Z$.

    Fix $a\in K^\alg$, and note that $a\in\Delta_\gamma$ for all $\gamma > |v(a)|$. Then for all $\gamma>|v(a)|$, since $Kv_\gamma$ is algebraically closed, there exists $b_\gamma\in K$ such that $v_\gamma(a-b_\gamma) > v_\gamma(a) = 0$. In other words, $v(a-b_\gamma) > \gamma$, and thus $K$ is dense in $K^\alg$ with respect to $v$.
\end{proof}

In particular, this lemma applies when $K$ is IAC. Combining the previous two results, we obtain our criterion for identifying VAC fields.

\begin{prop} \label{criterion}
	Suppose $K$ is IAC. Then the following are equivalent:
	\begin{enumerate}
		\item $K$ is VAC.
		\item For every rank 1 valuation $v$ on $K^\alg$, $K$ is dense in $(K^\alg,v)$
		\item For every rank 1 valuation $v$ on $K$, the completion of $(K,v)$ is algebraically closed.
	\end{enumerate}
\end{prop}

\begin{proof}
	$(1)\to(2)$: If $K$ is VAC then $K$ is dense in $(K^\alg,v)$ for every valuation $v$ on $K^\alg$, including those of rank 1.
	
	$(2)\to(3)$: Let $v$ be a rank 1 valuation on $K^\alg$. Then $K$ is dense in $(K^\alg,v)$, so by Proposition \ref{prop-0.1}, the completion of $(K,v)$ is algebraically closed.
	
	$(3)\to(1)$: Fix a valuation $w$ on $K$; we wish to show that $K$ is dense in $K^\alg$ with respect to $w$. If $\cO_w$ does not have a maximum non-trivial coarsening, then this follows immediately from Lemma \ref{nomax}. Otherwise, let $v$ be the valuation corresponding to the maximum non-trivial coarsening of $\cO_w$; by choice of $v$, it has no proper non-trivial convex subgroups, and hence has rank 1. Then the completion of $(K,v)$ is algebraically closed by (3), and so by Proposition \ref{prop-0.1}, $K$ is dense in $K^\alg$ with respect to $v$. But $v$ and $w$ induce the same topology on $K^\alg$, which means $K$ is also dense in $K^\alg$ with respect to $w$.
\end{proof}

We will use this criterion to establish two situations in which IAC and VAC are equivalent. The argument in each of these situations is based on the following result of Macintyre, McKenna, and van den Dries:

\begin{fact} \label{lemma7} \cite[Lemma 7]{MMv}
    Let $(K,v)$ be a perfect henselian field such that
    \begin{enumerate}
        \item $Kv$ is algebraically closed,
        \item $vK$ is divisible,
        \item if $\Char(K) = 0$ and $\Char(Kv) = p>0$ then $K^\times$ is $p$-divisible, and
        \item if $\Char(K) = p>0$ then $K$ is closed under Artin-Schreier extensions.
    \end{enumerate}
    Then $K$ is algebraically closed.
\end{fact}

\subsection{Positive Characteristic}
\label{positive}

In Section \ref{counterexample}, we observed that particular Artin-Schreier extensions can be used to demonstrate that an IAC field of positive characteristic is not VAC. As the theorem below shows, the existence of Artin-Schreier extensions is the only way that IAC fields can fail to be VAC in positive characteristic.

\begin{theorem} \label{asc}
    Suppose $K$ is field of positive characteristic which is IAC and Artin-Schreier closed. Then $K$ is VAC.
\end{theorem}

\begin{proof}
    By Proposition \ref{criterion}, it suffices to show that for every rank 1 valuation $v$ on $K$, the completion of $K$ with respect to $v$ is algebraically closed. We will establish this through Fact \ref{lemma7}.
	
	Fix a rank 1 valuation $v$ on $K$, and let $L$ be the completion of $K$ with respect to $v$. As remarked on page 85 of \cite{EP05}, the completion of every rank 1 valued field is henselian. Moreover, since $K$ is IAC and $L$ is an immediate extension of $K$, we get that $Lv = Kv$ is algebraically closed and $vL = vK$ is divisible. Lastly, $L$ is closed under Artin-Schreier extensions by Lemma 4.8 of \cite{Kuhl10} and perfect by Corollary 4.7 of the same paper.
	
	Thus, we may apply Fact \ref{lemma7} to obtain that $L$ is algebraically closed, and hence that $K$ is VAC.
\end{proof}

From this result, we can see that for fields without Artin-Schreier extensions, IAC and VAC are equivalent. One such case is that of NIP fields:

\begin{cor}
    Suppose $K$ is an infinite NIP field of positive characteristic. Then $K$ is IAC if and only if it is VAC.
\end{cor}

\begin{proof}
    Suppose $K$ is NIP and IAC. By \cite{KSW}, every infinite NIP field is Artin-Schreier closed, so by the theorem, $K$ is VAC. The converse always holds by Proposition \ref{easy}(3).
\end{proof}

In general, it is unclear whether being Artin-Schreier closed is a necessary condition for an IAC field to be VAC. In light of Proposition \ref{no-asc}, Artin-Schreier defect extensions certainly need to be avoided, but the existence of a VAC field with a defectless Artin-Schreier extension is currently an open problem.

\subsection{Characteristic Zero}
\label{zero}

Since an IAC field $K$ of characteristic zero will have valuations of every possible residue characteristic, in order to apply Fact \ref{lemma7} we need to know that the multiplicative group is $p$-divisible for every prime $p$. But just like in the proof of Theorem \ref{asc}, we need to apply Fact \ref{lemma7} not to $K$ itself, but to the completion of $K$ with respect to some valuation. We begin with a straightforward lemma that shows that $p$-divisibility of the multiplicative group passes from any valued field to its completion.

\begin{lemma} \label{p-divisible}
	Suppose $(K,v)$ is a valued field such that $K^\times$ is $p$-divisible. Then $L^\times$ is $p$-divisible, where $L$ is the completion of $K$ with respect to $v$.
\end{lemma}

\begin{proof}
	Consider some extension of $v$ to $L^\alg$ which we will also denote $v$. Fix $a\in L$, and let $b_1,\ldots,b_p\in L^\alg$ be the roots of $X^p-a$. We will show that one of these roots can be approximated arbitrarily well in $K$, and hence is contained in $L$.
	
	Let $(\gamma_\alpha)_{\alpha<\kappa}$ be a cofinal sequence in $vK = vL$. By continuity of roots, for each $\alpha<\kappa$ there exists $c_\alpha \in K$ such that each root of $X^p-c_\alpha$ is within $\gamma_\alpha$ of a root of $X^p-a$. By assumption, at least one of the roots of $X^p-c_\alpha$ is in $K$; call this root $d_\alpha$. While it is not necessarily true that each $d_\alpha$ approximates the same root of $X^p-a$, by the pigeonhole principle, there is $1\leq i\leq p$ and a cofinal subsequence of $(d_\alpha)_{\alpha<\kappa}$ such that $v(d_\alpha-b_i) > \gamma_\alpha$ for each $d_\alpha$ in the subsequence. Thus, $b_i \in L$, and so $L^\times$ is $p$-divisible.
\end{proof}

\begin{theorem} \label{divisible}
    Suppose $K$ is an IAC field of characteristic zero such that $K^\times$ is divisible. Then $K$ is VAC.
\end{theorem}

\begin{proof}
	By Proposition \ref{criterion}, it suffices to show that the completion of $K$ with respect to any rank 1 valuation is algebraically closed. Let $v$ be a rank 1 valuation of $K$ and let $L$ be the completion of $K$ with respect to $v$. As in Theorem \ref{asc}, $L$ is IAC and henselian.

    If $\Char(Lv) = 0$ then by Ax-Kochen-Ershov, $L$ must be algebraically closed. On the other hand, if $\Char(Lv) = p > 0$ then $L^\times$ is $p$-divisible by Lemma \ref{p-divisible}, and hence $L$ is algebraically closed by Fact \ref{lemma7}. In either case, $K$ is VAC by Proposition \ref{criterion}.
\end{proof}

Unlike the assumption in Theorem \ref{asc} that $K$ is Artin-Schreier closed, we know that the assumption that $K^\times$ is divisible in Theorem \ref{divisible} is stronger than necessary. As observed by Hong in his thesis \cite{Hong13}, every archimedean real closed field is VAC, but such fields clearly do not have $2$-divisible multiplicative groups. We repeat Hong's proof below.

\begin{prop} \label{realex}
    Every archimedean real closed field $R$ is VAC.
\end{prop}

\begin{proof}
    Let $C = R[i]$ be the algebraic closure of $R$, and fix a non-trivial valuation $v$ on $C$. Let $\gamma\in vC$ and $a = x+yi \in C$ with $x,y\in R$; we want to find $b\in R$ with $v(b-a) > \gamma$.

    Recall that a valuation on an ordered field is called convex if $v(x) > 0$ implies $|x| < \frac1n$ for all $n\in\N$. It follows easily that $R$ cannot have a non-trivial convex valuation, since $v(x) < 0$ would imply $|x| > n$ for all $n\in\N$, contradicting the archimedean property of $R$. Thus, $v$ is not a convex valuation, which means $R$ must contain an element $e$ with $v(e) > 0$ and $e > 1$.

    Let $c\in R$ be any element with $c>0$ and $v(c) > 2\gamma$. Then, since $R$ is archimedean, there exists $n\in\N$ such that $e^nc > y^2$. Thus there exists $d\in R$ such that $(d-x)^2 = e^nc-y^2 > 0$, because $R$ is real closed. Rearranging, this means that $(d-x)^2+y^2 = e^nc$, and hence
    \[ 2\gamma < v(e^nc) = v((x-d)^2+y^2) = v(d-x-iy) + v(d-x+iy) .\]
    One of the valuations on the right must therefore be greater than $\gamma$. If $v(d-x-iy) > \gamma$ then we may simply take $b=d$. Otherwise,
    \[ \gamma < v(d-x+iy) = v((d-2x)+(x+iy)) = v((2x-d)-(x+iy)) \]
    since $v(z) = v(-z)$. In this case, we may take $b = 2x-d\in R$, completing the proof.
\end{proof}

Hong's result can be very easily extended to show that for real closed fields, IAC is equivalent to VAC. As with the previous theorems, the assumption that $R$ is real closed is likely stronger than necessary.

\begin{theorem} \label{realthm}
	Let $R$ be a real closed field. Then the following are equivalent:
	\begin{enumerate}
		\item $R$ is archimedean
		\item $R$ is VAC
		\item $R$ is IAC
	\end{enumerate}
\end{theorem}

\begin{proof}
    $(1)\to (2)$: This is precisely Proposition \ref{realex}.

    $(2)\to (3)$: This holds even without the assumption that $R$ is real closed by Proposition \ref{easy}(3).

	$(3)\to (1)$: It is easy to check that the convex hull of $\Z$ in any ordered field is a valuation ring. By Proposition 2.2.4 of \cite{EP05}, this valuation ring has a formally real residue field. Since $R$ is IAC by assumption, this can only happen if the corresponding valuation is trivial. Then the convex hull of $\Z$ is all of $R$, which means $R$ is archimedean.
\end{proof}

\section{Strongly IAC Fields}
\label{52-stronglyiac}

The previous sections focused on IAC and VAC fields as algebraic objects. In this section, we consider some basic model theoretic properties of these fields. Throughout this section, we consider both pure fields in the language $\Lring = \{0,1,+,-,\cdot\}$ and valued fields in the language $\Ldiv = \Lring\cup\{\mid\}$, where $\mid$ is a binary relation interpreted as $x\mid y$ if and only if $v(x) \leq v(y)$ for some distinguished valuation $v$. To help avoid confusion, we will write $K$ for pure fields ($\Lring$-structures) and $(K,v)$ or $(K,\cO)$ for valued fields ($\Ldiv$-structures).

One curious difference between IAC and VAC is that the definition of IAC can be made without specifying any valuations on $K^\alg$, whereas VAC seems to require quantifying over all valuations of $K^\alg$, not just valuations of $K$. However, in Proposition \ref{prop-0.1}, we saw that the density of $(K,v)$ in $(K^\alg,w)$ depends only on $v$, not on $w$. Moreover, as we prove below, for a fixed valuation $v$ this statement can be made in a first-order way in $\Ldiv$.

\begin{prop}
    The theory of valued fields that are dense in their algebraic closure is axiomatizable in $\Ldiv$.
\end{prop}

\begin{proof}
    For each $n$, let $\sigma_n$ be the formula
	\[ \forall a_0,\ldots,a_n \forall r \exists x\ \big( r\neq 0 \wedge a_n\neq 0 \wedge (a_nx^n + \ldots + a_0\text{ is separable}) \ \to\ v(a_nx^n + \ldots + a_0 ) > v(r) \big)  .\]
	In other words, if $(K,v)$ is a valued field then $\sigma_n$ asserts that for all $\gamma\in vK$ and all separable polynomials $f(X)\in K[X]$ of degree $n$, there exists $x\in K$ with $v(f(x)) > \gamma$. Then the conjunction $\bigwedge_{n=1}^\infty \sigma_n$ is a first-order restatement of Proposition \ref{prop-0.1}(5), which is equivalent to $K$ being dense in its algebraic closure with respect to any extension of $v$. Thus, the desired axiomatization is the union of the axioms for valued fields with $\{\sigma_n : n\in\N\}$.
\end{proof}

Discussing density in a first order way requires adding the valuation to the language, as in the proposition above. In general, IAC and VAC are not first order properties in the language of rings. For example, $\R$ is both IAC and VAC by Proposition \ref{realex}, but every real closure of $\R(t)$ is non-archimedean, and hence neither IAC nor VAC by Theorem \ref{realthm}.

One way to interpret this is that $\R$ is only IAC because it is a small model of its theory. Similar issues arise with the definitions of minimality and $P$-minimality; there are small structures that are minimal, for example, but have elementary extensions that are not. We avoid cases like this in the same way as those classes:

\begin{defn}
    We say that a field $K$ is strongly IAC if every field elementarily equivalent to $K$ (in the language of rings) is IAC.
\end{defn}

\begin{theorem} \label{strongiac}
    Let $K$ be a strongly IAC field, and fix a distinguished valuation ring $\cO$ of $K$. Then $(K,\cO)$ is dense in its algebraic closure.
\end{theorem}

\begin{proof}
    Consider a chain $\K = \K_0\preceq \K_1\preceq \ldots$ of elementary extensions $\K_n = (K_n,\cO_n)$ of $\K = (K,\cO)$ such that each $\K_{n+1}$ is $|K_n|^+$-saturated. Then $\K' = \bigcup_n \K_n$ is an elementary extension of $\K$ with valuation ring $\cO' = \bigcup_n \cO_n$; write $v$ for the valuation corresponding to $\cO'$ as well as its restriction to each $K_n$. Then for each $n$, $\K_n$ contains a realization of the partial type $\{v(x) > v(a) : a\in K_{n-1}\},$ and so there is a proper convex subgroup $\Delta_n$ of $vK_n$ which contains $vK_{n-1}$.
	
	By taking the convex hull of each $\Delta_n$ in $vK'$, we obtain an infinite chain of convex subgroups of $vK'$ with no maximum element. The union of this chain is $vK'$, and so the correspondence between convex subgroups and coarsenings tells us that $\cO'$ does not have a maximum non-trivial coarsening. Thus $(K',\cO')$ is dense in its algebraic closure by Lemma \ref{nomax}, and so by elementary equivalence, $(K,\cO)$ is dense in $K^\alg$.
\end{proof}

\begin{cor} \label{strongly-vac}
    Every strongly IAC field is VAC.
\end{cor}

\begin{proof}
    By the theorem, if $K$ is strongly IAC then it is dense in its algebraic closure with respect to every valuation, and hence is VAC.
\end{proof}

As mentioned in the introduction, Krupi\'nski has shown that every superrosy field of positive characteristic is IAC, and hence all such fields are strongly IAC. Then by Theorem \ref{strongiac}, every superrosy field of positive characteristic is VAC.

There are a number of relationships, both proven and conjectured, between the algebraic and model theoretic structures of fields. Two well-known conjectures are:

\begin{conj} 
	Every simple field is pseudo-algebraically closed.
\end{conj}

\begin{conj} 
	Every stable field is separably closed.
\end{conj}

Recently, there has been a large amount of focus and progress towards a related conjecture about strongly dependent fields, originally proposed in \cite{She14}; the paper \cite{HHJ18} discusses all three conjectures and some of the relationships between them.

As we know that all pseudo-algebraically closed fields and all separably closed fields are VAC, we could break Conjectures 5.5 and 5.6 into smaller, more manageable pieces:

\begin{conj}
	Let $K$ be a field whose theory is simple. Then
	\begin{enumerate}
		\item $K$ is strongly IAC, and
		\item If $K$ is VAC then $K$ is pseudo-algebraically closed.
	\end{enumerate}
\end{conj}

\begin{conj}
	Let $K$ be a field whose theory is stable. Then
	\begin{enumerate}
		\item $K$ is strongly IAC, and
		\item If $K$ is VAC then $K$ is separably closed.
	\end{enumerate}
\end{conj}

By Corollary \ref{strongly-vac}, Conjectures 5.5 and 5.7 are equivalent, as are Conjectures 5.6 and 5.8. However, the smaller components above may be easier to prove than the full conjecture.

\bibliographystyle{alpha}
\bibliography{../References}

\end{document}